\pdfoutput=1
\documentclass[11pt, reqno]{amsart}

\numberwithin{equation}{section}

\usepackage{amsmath,amsfonts,amssymb}
\usepackage{mathtools}
\usepackage[utf8]{inputenc}
\usepackage[english]{babel}
\usepackage{tikz-cd}
\usepackage[colorlinks=true,linkcolor=black,anchorcolor=black,
citecolor=black,filecolor=black,menucolor=black,runcolor=black,
urlcolor=black]{hyperref}
\usepackage{cleveref}
\usepackage[margin=1.2in,footskip=0.25in]{geometry}

\DeclareUnicodeCharacter{FB01}{fi}
\newtheorem{theorem}{Theorem}[section]
\newtheorem{corollary}[theorem]{Corollary}
\newtheorem{lemma}[theorem]{Lemma}
\newtheorem{prop}[theorem]{Proposition}

\newcommand{\Hom}{\operatorname{Hom}}
\newcommand{\Ext}{\operatorname{Ext}}

\DeclareMathOperator{\Aut}{Aut}

\DeclareMathOperator{\rank}{rank}

\DeclareMathOperator{\GL}{GL}
\DeclareMathOperator{\PGL}{PGL}

\theoremstyle{definition}
\newtheorem{definition}{Definition}[section]
\newtheorem{remark}[definition]{Remark}
\begin{document}
	\title[Topology of moduli of parabolic connections with fixed determinant]{Topology of moduli of parabolic connections with fixed determinant}
	
\author[N. Das]{Nilkantha Das}
\address{Department of Mathematics, Birla Institute of Technology Mesra, Mesra, Ranchi 835215, India}
\email{nilkantha.das@bitmesra.ac.in}

	\author[S. Roy]{Sumit Roy}
 \address{Stat-Math Unit, Indian Statistical Institute, 203 B.T. Road, Kolkata 700 108, India.}
\email{sumitroy\_r@isical.ac.in}

	
	\begin{abstract}
		Let $X$ be a compact Riemann surface of genus $g \geq 2$ and $D\subset X$ be a fixed finite subset. Let $\xi$ be a line bundle of degree $d$ over $X$. Let $\mathcal{M}(\alpha, r, \xi)$ (respectively, $\mathcal{M}_{\mathrm{conn}}(\alpha, r, \xi)$) denote the moduli space of stable parabolic bundles (respectively, parabolic connections) of rank $r$ $(\geq 2)$, determinant $\xi$ and full flag generic parabolic weight type $\alpha$. We show that $
  \pi_k(\mathcal{M}_{\mathrm{conn}}(\alpha, r, \xi)) \cong \pi_k(\mathcal{M}(\alpha, r, \xi)) $ for $k \leq2(r-1)(g-1)-1$. As a consequence, we deduce that the moduli space $\mathcal{M}_{\mathrm{conn}}(\alpha, r, \xi)$ is simply connected. We also show that the Hodge structures on the torsion-free parts of both the cohomologies $H^k(\mathcal{M}_{\mathrm{conn}}(\alpha, r, \xi),\mathbb{Z})$ and $H^k(\mathcal{M}(\alpha, r, \xi),\mathbb{Z})$ are isomorphic for all $k\leq 2(r-1)(g-1)+1$.
	\end{abstract}
	\maketitle
	
	\section{Introduction}

 In this article, we consider the moduli space of parabolic connections over a compact Riemann surface $X$ of genus $g \geq 2$, which is the moduli space of parabolic $\Lambda$-modules (see \cite{A17}), where $\Lambda$ is the sheaf of differential operators. In \cite{I13}, Inaba also constructed the moduli space of parabolic connections over a compact Riemann surface.

 Let $D \subset X$ be a fixed finite subset of $X$. A \textit{parabolic vector bundle} $E_*$ over $X$ is a holomorphic vector bundle $E$ with a weighted filtration over each point of $D$. A \textit{parabolic connection with fixed determinant $\xi$} is a pair $(E_*,\nabla)$  such that $E_*$ is a parabolic bundle with determinant $\xi$  and $\nabla: E \longrightarrow E \otimes K(D)$ is a logarithmic connection on $E$, where $K$ is the canonical line bundle over $X$. 

Next we consider the moduli space of semistable parabolic connections over $X$ (see \cite{B02, I13, S22}). This moduli space has singularities in general, but if we consider that the weights are generic, the moduli space is smooth. In particular, the moduli space of stable parabolic connections with generic weight is smooth. We consider the full flag filtration over each point of $D$. Here the generic set of weights is given by $\alpha =\{\alpha_i(p) \in [0,1) \  \big\vert \  1\leq i \leq r,\  p \in D\}$, where $r$ denotes the rank of the vector bundle.

 Let $\xi$ be a line bundle over $X$ of degree $d$ and let $\mathcal{M}(\alpha, r, \xi)$ denote the moduli space of stable parabolic bundles of rank $r$ $(\geq 2)$, determinant $\xi$ and full flag generic weights $\alpha$. It is a smooth projective rational complex variety (cf. \cite{BY99}). As a consequence, $\mathcal{M}(\alpha, r, \xi)$ is simply connected (cf. \cite[XI, Corollary 1.2]{SGA1}, and \cite[Corollary 4.18]{D01}). In this article, we are interested in $\mathcal{M}_{\mathrm{conn}}(\alpha, r, \xi)$, the moduli space of stable parabolic connections of fixed determinant $\xi$. Consider the forgetful map
 \[
 \pi : \mathcal{U}_{\mathrm{conn}}(\alpha, r, \xi) \longrightarrow \mathcal{M}(\alpha, r, \xi),
 \]
 where $\mathcal{U}_{\mathrm{conn}}(\alpha, r, \xi) \subset \mathcal{M}_{\mathrm{conn}}(\alpha, r, \xi)$ denotes the open subset containing all stable parabolic connections $(E_*,\nabla)$ whose underlying parabolic bundle $E_*$ is stable (see \ref{forgetful}). The map $\pi$ is, in fact, shown to be a surjective equidimensional map with connected and contractible fibers. Thus $\mathcal{U}_{\mathrm{conn}}(\alpha, r, \xi)$ deformation retracts to $\mathcal{M}_{\mathrm{conn}}(\alpha, r, \xi)$. In \Cref{dimensioncomputation}, the codimension of the complement of $\mathcal{U}_{\mathrm{conn}}(\alpha, r, \xi)$ in $\mathcal{M}_{\mathrm{conn}}(\alpha, r, \xi)$ is computed. It is shown to be large enough to conclude
 \[
  \pi_k(\mathcal{M}_{\mathrm{conn}}(\alpha, r, \xi)) \cong \pi_k(\mathcal{M}(\alpha, r, \xi))
  \]
 for $1\leq k \leq 2(r-1)(g-1)$ (see Theorem \ref{main}).
 As a consequence, the moduli space $\mathcal{M}_{\mathrm{conn}}(\alpha, r, \xi)$ is simply connected (see Corollary \ref{cor}). We also prove that if either $r\geq 3$ or $g\geq3$,  then the natural map $$\pi_2(\mathcal{M}(\alpha, r, \xi))\longrightarrow H_2(\mathcal{M}(\alpha, r, \xi))$$ induced by the Hurewicz map of $\mathcal{M}(\alpha, r, \xi)$ is an isomorphism. 
 
In \Cref{mhs}, we consider the Hodge structures of the cohomology groups of $\mathcal{M}_{\mathrm{conn}}(\alpha, r, \xi)$ with $\mathbb{Z}$-coefficients. In Theorem \ref{mhstheorem}, we show that mixed Hodge structure of torsion-free part $H^i(\mathcal{M}_{\mathrm{conn}}(\alpha, r, \xi),\mathbb{Z})/\mathrm{torsion}$ is isomorphic to the Hodge struture of $H^i(\mathcal{M}(\alpha, r, \xi),\mathbb{Z})/\mathrm{torsion}$ for all $i\leq 2(r-1)(g-1)+1$. Since the moduli space $\mathcal{M}(\alpha, r, \xi)$ of stable parabolic bundles with fixed determinant $\xi$ is smooth projective, the Hodge structure on $H^i(\mathcal{M}(\alpha, r, \xi),\mathbb{Z})/\mathrm{torsion}$ is pure of weight $i$. Thus, it follows that the Hodge structure on  $H^i(\mathcal{M}_{\mathrm{conn}}(\alpha, r, \xi),\mathbb{Z})/\mathrm{torsion}$ is pure of weight $i$ for all $i\leq 2(r-1)(g-1)+1$.
 
\section{Preliminaries}
 In this section, we first recall some definitions regarding parabolic bundles over a curve, see \cite{MY92,MS80} for more details.

 Let $X$ be a compact Riemann surface and let $D = \{p_1,\dots, p_n\} $ be a subset of $n$ distinct marked points of $X$. The subset $D$ is fixed throughout the note. We will make an abuse of notation denoting the divisor $\sum \limits_{i=1}^n p_i$ by the same letter $D$. 
	\subsection{Parabolic bundles}
	\begin{definition}\label{parabolic}
		A \textit{parabolic bundle} $E_*$ on $X$ is a holomorphic vector bundle $E$ of rank $r$ $(\geq 2)$ on $X$ together with a parabolic structure along $D$, i.e. for every point $p \in D$, we have
		\begin{enumerate}
			\item a strictly decreasing filtration of subspaces of the fiber 
			\[
			E_p = E_{p,1}\supsetneq E_{p,2} \supsetneq \dots \supsetneq E_{p,r_p} \supsetneq E_{p,r_p+1} =\{0\},
			\]
			\item a sequence of real numbers (parabolic weights) satisfying 
			\[
			0\leq \alpha_1(p) < \alpha_2(p) < \dots < \alpha_{r_p}(p) < 1,
			\]
		\end{enumerate}
		where $r_p \in \{1,\dots,r\}$. 
	\end{definition}
	
	The collection of all such weights $\{(\alpha_1(p),\alpha_2(p),\dots,\alpha_{r_p}(p))\}_{p\in D}$ is denoted by $\alpha$ for a fixed parabolic structure. The collection $\alpha$ is said to have \textit{full flags} if 
	\[\mathrm{dim}(E_{p,i}/E_{p,i+1}) = 1
	\] 
	for every $1\leq i \leq r_p$ and for every point $p\in D$, or equivalently $r_p=r$ for every $p\in D$.
	
	The \textit{parabolic degree} of a parabolic bundle $E_*$ with set of weights $\{\alpha_i(p)\}_{p\in D}$ is defined as
	\[
	\operatorname{pardeg}(E_*) \coloneqq \deg(E)+ \sum\limits_{p\in D}\sum\limits_{i=1}^{r_p} \alpha_i(p) 
	\]
	(cf. \cite[Definition $1.11$, p.$214$]{MS80}, \cite[p. $78$]{MY92}), and the real number
	\[
	\mu_{\mathrm{par}}(E_*) \coloneqq \frac{\text{pardeg}(E_*)}{\mathrm{rank}(E)} \in \mathbb{R}
	\]
	is called the \textit{parabolic slope}.

    Consider any holomorphic subbundle $F \subset E$. For every $p \in D$, by intersecting the filtration of subspaces of $E_p$ with $F_p$, we can obtain a filtration of subspaces of the fiber $F_p$. The weight of a subspace $A \subset F_p$ that appears in this filtration is given by
    \[
    \mathrm{max}\{\alpha_i(p) \mid A \subset E_{p,i} \cap F_p\}.
    \]
    It gives an induced parabolic structure on the subbundle $F$ and we denote the parabolic subbundle by $F_*$.
	
\begin{definition}
 A parabolic vector bundle $E_*$ is called \textit{stable} (respectively, \textit{semistable}) if for every nonzero proper holomorphic subbundle $F \subset E$ the following inequality holds:
	\[
	\mu_{\mathrm{par}}(F_*) < \mu_{\mathrm{par}}(E_*) \hspace{0.2cm} (\mathrm{respectively, } \hspace{0.2cm} \mu_{\mathrm{par}}(F_*) \leq \mu_{\mathrm{par}}(E_*) ).
	\]  
\end{definition}
	
	We denote by $\mathcal{M}(\alpha, r, d)$ the moduli space of semistable parabolic bundles over $X$ of rank $r$, degree $d$ and parabolic structure $\alpha$. In \cite{MS80}, Mehta and Seshadri constructed this moduli space using Mumford's geometric invariant theory. They also showed that it is a normal complex projective variety of dimension
	\[
	\dim\mathcal{M}(\alpha, r, d) = r^2(g-1) + 1 + \dfrac{n(r^2-r)}{2}, 
	\]
	where the last summand of the right-hand side comes from the assumption that the filtration over each point $p\in D$ has full flags. Moreover, a weight type $\alpha$ is called \textit{generic} if semistability implies stability for such weights. The stable locus of $\mathcal{M}(\alpha, r, d)$ is contained in the smooth locus of the moduli space. So, $\mathcal{M}(\alpha, r, d)$ is a smooth variety if $\alpha$ is chosen to be generic. From now onward, we assume that the parabolic weights are generic and the parabolic structure has full flags at each parabolic point $p\in D$.

 	Let $\xi$ be a line bundle over $X$ of degree $d$. Consider the determinant map from $\mathcal{M}(\alpha, r, d)$ to $\mathrm{Pic}^d(X)$, the space of degree $d$ line bundles over $X$:
  \begin{align*}
      \det : \mathcal{M}(\alpha, r, d) &\longrightarrow \mathrm{Pic}^d(X) \\
      E_* &\mapsto \det(E),
  \end{align*}
   where $E$ is the underline vector bundle of the parabolic bundle $E_*$.
  Then, 
  \[
  \mathcal{M}(\alpha, r, \xi) \coloneqq \mathrm{det}^{-1}(\xi)
  \]
  denotes the moduli space of stable parabolic bundles over $X$ of fixed determinant $\xi$. It is a smooth complex projective variety of dimension
	\begin{align*}
	    \dim \mathcal{M}(\alpha, r, \xi) &= \dim \mathcal{M}(\alpha, r, d) - g\\
     &= (g-1)(r^2-1) + \dfrac{n(r^2-r)}{2}.
	\end{align*}

 \begin{definition}
		A \textit{parabolic morphism} $\varphi : E_* \to E^\prime_*$ between two parabolic bundles $E_*$ and $E_*'$ is a morphism of underlying vector bundles $E$ and $E'$ which satisfies
		the following: at every $p \in D$, we have 
		\[
		\alpha_i(p) > \alpha_j^\prime(p) \implies \varphi(E_{p,i}) \subseteq E_{p,j+1}^\prime.
		\]
		Such a morphism is called \textit{strongly parabolic} if
		\[ \alpha_i(p) \geq \alpha_j^\prime(p) \implies \varphi(E_{p,i}) \subseteq E_{p,j+1}^\prime
		\]
		at every $p \in D$.
	\end{definition}
 We denote the sheaf of parabolic endomorphisms of $E_*$ by $\mathrm{PEnd}(E_*)$. Similarly, the sheaf of strongly parabolic endomorphisms will be denoted by $\mathrm{SPEnd}(E_*)$. For a parabolic bundle $E_*$, we will denote the sheaf of traceless parabolic endomorphisms (respectively, strongly parabolic endomorphisms) by $\mathrm{PEnd}_0(E_*)$ (respectively, $\mathrm{SPEnd}_0(E_*)$). In this context, we will use the parabolic version of the Serre duality, i.e.,
  \[
  H^1(X, \mathrm{PEnd}_0(E_*)) \cong H^0(X, \mathrm{SPEnd}_0(E_*) \otimes K(D))^\vee,
  \]
where $K$ is the canonical line bundle over $X$ and $K(D) = K \otimes \mathcal{O}(D)$.
	
	\subsection{Parabolic connections}
 
	A \textit{logarithmic connection} on a holomorphic vector bundle $E$ over $X$, singular over $D$ is a holomorphic differential operator
	\[
	\nabla : E \to E \otimes K(D) 
	\]
	satisfying the Leibniz identity
	\begin{equation}\label{leibniz}
	    \nabla(fs)= f\nabla(s) + df\otimes s,
	\end{equation}
	where $f$ is a locally defined holomorphic function on $X$ and $s$ is a locally defined holomorphic section of $E$. For more details about the logarithmic connection, see \cite{D70} and \cite{BGKHME87}. 

 By the Poincaré adjunction formula \cite[p. $146$]{GH78}, the fiber of $K(D)$ over a point $p\in D$ is identified with $\mathbb{C}$. To explain this in detail, consider the following morphism
 \begin{align*}
 \mathbb{C} &\longrightarrow \left.K(D)\right|_p \\ 
 x &\mapsto x\cdot \frac{dz}{z}(p),
 \end{align*}
where $z$ is coordinate function on $X$ defined on an open neighbourhood of $p\in D$ such that $z(p)=0$. It gives us the required isomorphism.
Let $E$ be a holomorphic vector bundle with a logarithmic connection $\nabla$. From \eqref{leibniz}, we get that the following composition of morphisms
\[
E \xlongrightarrow[]{\nabla} E \otimes K(D) \xlongrightarrow[]{} (E \otimes K(D))_p \xlongrightarrow[]{\sim} E_p
\]
is $\mathcal{O}_X$-linear. The isomorphism $(E \otimes K(D))_p \xlongrightarrow[]{\sim} E_p$ follows from the fact that $K(D)_p \cong \mathbb{C}$. Thus, we get a $\mathbb{C}$-linear map
\[
\mathrm{Res}(\nabla,p) : E_p \longrightarrow E_p.
\]
We call this $\mathbb{C}$-linear endomorphism the \textit{residue} of $\nabla$ at $p\in D$ (see \cite{D70}). Following \cite[Theorem $3$]{O82}, a logarithmic connection $\nabla$, singular over $D$, satisfies
\begin{equation}\label{residue}
    \deg(E) + \sum_{p\in D} \mathrm{trace}\big(\mathrm{Res}(\nabla, p)\big) = 0,
\end{equation}
where $\mathrm{trace}(\mathrm{Res}(\nabla, p))$ denotes the trace of the residue $\mathrm{Res}(\nabla, p) \in \mathrm{End}(E_p)$, for all $p \in D$.

	\begin{definition}\label{pconn}
		A \textit{parabolic connection} on $E_*$ over $X$ is a logarithmic connection $\nabla$ on $E$, singular over $D$, satisfying the following conditions:
		\begin{enumerate}
			\item For all $p \in D$, the logarithmic connection $\nabla$ on the fiber $E_p$ satisfies 
			\[
			\nabla(E_{p,i}) \subseteq E_{p,i} \otimes \left.K(D)\right|_p
			\]
			for all $i \in \{1,2,\dots ,r\}$.
			\item For every $p \in D$ and for every $i \in \{1,\dots , r\}$ the action of the residue $\mathrm{Res}(\nabla,p) \in \mathrm{End}(E_p)$ on the quotient $E_{p,i}/E_{p,i+1}$ is the multiplication by $\alpha_{i}(p)$, where $\alpha_{i}(p)$'s are the weights over $p \in D$. Since the residue $\mathrm{Res}(\nabla,p)$ preserves the filtration over every $p\in D$, it acts on every quotient space.
		\end{enumerate}
	\end{definition}	
	We will denote a parabolic connection by $(E_*,\nabla)$. For a holomorphic subbundle $F \subset E$, the induced parabolic subbundle $F_* \subset E_*$ is called $\nabla$\textit{-invariant} if $\nabla(F) \subseteq F \otimes K(D)$.

    \begin{remark}\label{degreezero}
        \noindent
Note that by Definition \ref{pconn}, $$\operatorname{trace}(\operatorname{Res}(\nabla, p)) = \sum_{i=1}^{r} \alpha_{i}(p)$$ for each $p$. Substituting into equation \eqref{residue}, we obtain
\[
\deg(E) + \sum_{p \in D} \sum_{i=1}^{r} \alpha_{i}(p) = 0,
\]
that is, $\operatorname{pardeg}(E_*) = 0$.

    \end{remark}
	
	\begin{definition}
		A parabolic connection $(E_*,\nabla)$ is said to be \textit{stable} (respectively, \textit{semistable}) if for every nonzero proper $\nabla$-invariant parabolic subbundle $F_* \subset E_*$, the following inequality holds:
		\[
		\mu_{\mathrm{par}}(F_*) < \mu_{\mathrm{par}}(E_*) \hspace{0.3cm} (\mathrm{respectively,}\hspace{0.05cm} \mu_{\mathrm{par}}(F_*) \leq \mu_{\mathrm{par}}(E_*)).
		\]
	\end{definition}
 Note that the stability of a parabolic connection $(E_*,\nabla)$ doesn't imply the stability of the underlying parabolic bundle $E_*$. Let $\mathcal{M}_{\mathrm{conn}}(\alpha,r, d)$ denote the moduli space of stable parabolic connections of rank $r$, degree $d$ and weight type $\alpha$ (assuming full flag structure). From \cite[Theorem $2.1$]{I13}, $\mathcal{M}_{\mathrm{conn}}(\alpha,r, d)$ is a smooth quasi-projective irreducible variety over $\mathbb{C}$ of dimension
\[
\dim\mathcal{M}_{\mathrm{conn}}(\alpha, r, d) = 2\dim\mathcal{M}(\alpha, r, d) = 2r^2(g-1) + 2 + n(r^2-r).
\]
Similarly, the moduli space  $\mathcal{M}_{\mathrm{conn}}(\alpha,r, \xi)$ of stable parabolic connections of rank $r$, fixed determinant $\xi$ and weight type $\alpha$ is a smooth quasi-projective irreducible complex variety of dimension 
	\[
	\dim\mathcal{M}_{\mathrm{conn}}(\alpha, r, \xi)= 2\dim \mathcal{M}(\alpha, r, \xi) = 2(g-1)(r^2-1) + n(r^2-r).
	\]

	\section{Dimension Computations}\label{dimensioncomputation}
 In this section, our goal is to compute the dimensions of certain subsets of the moduli space $\mathcal{M}_{\mathrm{conn}}(\alpha, r, \xi)$. Let $$\mathcal{U}_{\mathrm{conn}}(\alpha, r, \xi) \subset \mathcal{M}_{\mathrm{conn}}(\alpha, r, \xi)$$ be the subset containing stable parabolic connections $(E_*,\nabla)$ whose underlying parabolic bundle $E_*$ is stable. It is an open subvariety of $\mathcal{M}_{\mathrm{conn}}(\alpha, r, \xi)$. Therefore, we have the following well-defined forgetful map
	\begin{equation}\label{forgetful}
	\pi : \mathcal{U}_{\mathrm{conn}}(\alpha, r, \xi) \longrightarrow \mathcal{M}(\alpha, r, \xi)
	\end{equation}
    that sends a parabolic connection $(E_*,\nabla)$ to $E_*$, that is, $\pi$ forgets the parabolic connection. Since $\operatorname{pardeg}(E_*)=0$ as discussed in Remark \ref{degreezero}, the surjectivity of $\pi$ follows from \cite[Theorem $3.1$]{B02}. For every $E_* \in \mathcal{M}(\alpha, r, \xi)$, we are interested in computing the dimensions of the fiber $\pi^{-1}(E_*)$.
	
	\begin{lemma}
		For every parabolic bundle $E_* \in \mathcal{M}(\alpha, r, \xi)$, the fiber $\pi^{-1}(E_*)$ forms an affine space over $H^0(X,\mathrm{SPEnd}_0(E_*)\otimes K(D))$.
	\end{lemma}
	\begin{proof}
		For any parabolic bundle $E_* \in \mathcal{M}(\alpha, r, \xi)$, the fiber $\pi^{-1}(E_*)$ consists of all connections on the parabolic bundle $E_*$ with fixed determinant $\xi$. Let $\nabla, \nabla' \in \pi^{-1}(E_*)$ be two conncections in the fiber $\pi^{-1}(E_*)$. Then their difference $\nabla - \nabla'$ is an $\mathcal{O}_X$-module morphism from the underlying vector bundle $E$ to $E\otimes K(D)$ such that the difference of residues
		\[
		\mathrm{Res}(\nabla,p) -\mathrm{Res}(\nabla',p) = \mathrm{Res}(\nabla-\nabla',p)
		\]
		acts on the quotients $E_p^i/E_p^{i+1}$ as the zero morphism for all $p\in D$. Thus, for every $p \in D$,
		\[
		(\nabla -\nabla')(E_p^i)\subset E_p^{i+1} \otimes \left.K(D)\right|_p.
		\]
		Also, since the determinant of both parabolic connections are the same, we get
		\[
		\mathrm{trace}(\nabla - \nabla') = 0.
		\]
  Conversely, take
  \[
  \theta \in H^0(X, \mathrm{SPEnd}_0(E_*)\otimes K(D))
  \]
  and consider the operator $\nabla + \theta$. It is a logarithmic connection on $E$ and satisfies both the conditions in Definition \ref{pconn}. Hence, the fiber $\pi^{-1}(E_*)$ which consists of all connections on the parabolic bundle $E_*$ with fixed determinant forms an affine space over $H^0(X,\mathrm{SPEnd}_0(E_*)\otimes K(D))$.
	\end{proof}

 		Let $\mathcal{P}\textit{Conn}(E_*)$ be the space that consists of all parabolic connections $\nabla$ on $E_*$ with fixed determinant $\xi$. We note that $\mathcal{P}\textit{Conn}(E_*)$ is an affine space for the vector space $H^0(X, \mathrm{SPEnd}_0(E_*) \otimes K(D))$, where $\mathrm{SPEnd}_0(E_*) \subset \mathrm{SPEnd}(E_*)$ is the subbundle defined by the sheaf of trace zero strongly parabolic endomorphisms. Therefore,
  \[
  \mathrm{rank} (\mathrm{SPEnd}_0(E_*)) = \mathrm{rank}(\mathrm{SPEnd}(E_*)) - 1.
  \]
  The quotient space $\mathcal{P}\textit{Conn}(E_*)/\mathrm{Aut}(E_*)$ parametrizes all  parabolic connections $\nabla$ on $E_*$ such that $(E_*,\nabla) \in \mathcal{M}_{\mathrm{conn}}(\alpha, r, \xi)$. 
  Our aim is to compute the dimension of the quotient space $\mathcal{P}\textit{Conn}(E_*)/\mathrm{Aut}(E_*)$. 
	
	\begin{lemma}\label{dimension}
  For any $(E_*,\nabla) \in \mathcal{M}_{\mathrm{conn}}(\alpha, r, \xi)$,
		\[
      \mathrm{dim} \big(\mathcal{P}\textit{Conn}(E_*)/\mathrm{Aut}(E_*)\big) = (r^2-1)(g-1) + \frac{n(r^2-r)}{2}.
  \]
	\end{lemma}
	\begin{proof}
 Since the Lie algebra of the automorphism group $\mathrm{Aut}(E_*)$ of the underlying parabolic bundle $E_*$ is isomorphic to $H^0(X, \mathrm{PEnd}(E_*))$, we get
  \begin{equation}\label{automorphism}
  \mathrm{dim} \mathrm{Aut}(E_*) = h^0(X, \mathrm{PEnd}_0(E_*)) +1.
  \end{equation}
  Since the pair $(E_*,\nabla)$ is stable, the isotropy for the natural action of the global automorphism group $\mathrm{Aut}(E_*)$ on $\mathcal{P}\textit{Conn}(E_*)$ is the subgroup defined by scalar automorphisms of $E_*$, i.e. of the form $\lambda \cdot \mathrm{Id}_{E_*}$ with $\lambda \in \mathbb{C}^*$. Therefore, the dimension of the space of all isomorphism classes of parabolic connections $\nabla$ on $E_*$ such that $(E_*,\nabla) \in \mathcal{M}_{\mathrm{conn}}(\alpha, r, \xi)$ is given by
  \begin{align}\label{dim1}
  \begin{split}
      \mathrm{dim} \big(\mathcal{P}\textit{Conn}(E_*)/\mathrm{Aut}(E_*)\big) &= h^0(X, \mathrm{SPEnd}_0(E_*) \otimes K(D)) - (\mathrm{dim}\mathrm{Aut}(E_*) - 1)\\
      &= h^1(X,\mathrm{PEnd}_0(E_*)) - h^0(X, \mathrm{PEnd}_0(E_*)) \\
      &= -\chi(\mathrm{PEnd}_0(E_*)).
      \end{split}
  \end{align}
  There is a skyscraper sheaf $\mathcal{F}_D$ supported on the set of parabolic points $D$ such that 
  \[
  0 \longrightarrow \mathrm{PEnd}_0(E_*) \longrightarrow \mathrm{End}_0(E) \longrightarrow \mathcal{F}_D \longrightarrow 0,
  \]
  is a short exact sequence of sheaves, where $\mathrm{End}_0(E)$ consists of trace zero endomorphisms of the underlying bundle $E$. Thus, we get
  \begin{equation}\label{euler}
  \chi(\mathrm{PEnd}_0(E_*)) =   \chi(\mathrm{End}_0(E)) - \chi(\mathcal{F}_D).
  \end{equation}
 Since $\mathrm{rank}(\mathrm{End}_0(E)) = \mathrm{rank}(\mathrm{End}(E)) -1 = r^2-1$ and $\deg \mathrm{End}_0(E) = 0$, the Riemann-Roch formula gives
  \begin{equation}\label{euler1}
      \chi(\mathrm{End}_0(E)) = (r^2-1)(1-g).
  \end{equation}
  Since we consider only full flag parabolic structure at each parabolic point, we get from \cite[Lemma $2.4$]{BH95} that
  \begin{equation}\label{euler2}
      \chi(\mathcal{F}_D) = \frac{n(r^2-r)}{2}.
  \end{equation}
  Thus, from \eqref{euler},\eqref{euler1} and \eqref{euler2}, we get
  \[
  \chi(\mathrm{PEnd}_0(E_*)) = (r^2-1)(1-g) - \frac{n(r^2-r)}{2}.
  \]
  Therefore, from \eqref{dim1}, we conclude that
  \begin{equation}
      \mathrm{dim} \big(\mathcal{P}\textit{Conn}(E_*)/\mathrm{Aut}(E_*)\big) = (r^2-1)(g-1) + \frac{n(r^2-r)}{2}.
  \end{equation}
This completes the proof.	
\end{proof}

\begin{corollary}\label{dimension1}
    For any $E_* \in \mathcal{M}(\alpha,r,\xi)$, the fiber $\pi^{-1}(E_*)$ has dimension
    \[
    \dim \pi^{-1}(E_*) = (r^2-1)(g-1) + \frac{n(r^2-r)}{2}.
    \]
      In particular, the dimension of the fibers $\pi^{-1}(E_*)$ doesn't depend on the parabolic bundle $E_*$.
\end{corollary}
\begin{proof}
    Note that $\pi^{-1}(E_*)$ is the same as $\mathcal{P}\textit{Conn}(E_*)/\mathrm{Aut}(E_*)$. The rest follows from Lemma \ref{dimension}.
\end{proof}

We now move towards the dimension computation of the complement of the open subvariety $\mathcal{U}_{\mathrm{conn}}(\alpha, r, \xi) $ inside $\mathcal{M}_{\mathrm{conn}}(\alpha, r, \xi)$. To do that, the numerical contributions coming from the parabolic weights in the destabilizing condition need to be controlled. This leads to an inequality involving the multiplicities and weights at each parabolic point in \Cref{codim}. The following lemma is essential for the required estimate.

\begin{lemma}\label{key_lemma}
Fix non-negative integers $r,r_1$ with $r_1 \leq r$. Let $m_i \in \{0,1\}$ for $1 \leq i \leq r$ such that $\sum\limits_{i=1}^r m_i=r_1$. Let $0\le \alpha_1<\cdots<\alpha_r<1$. Then
\[
\sum\limits_{i=1}^r (r\,m_i-r_1)\alpha_i
\leq
\sum_{j=1}^{r}(1-m_j)\Bigl(r_1-\sum_{i=1}^{j}m_i\Bigr).
\]
\end{lemma}

\begin{proof}
Set
\[
w:=\sum_{i=1}^r (r\,m_i-r_1)\alpha_i, \qquad
c:=\sum_{j=1}^{r}(1-m_j)\Bigl(r_1-\sum_{i=1}^{j}m_i\Bigr).
\]
For $1\le t\le r-1$, define $q_t:=\sum\limits_{i=1}^t m_i$ and $\Delta_t:=\alpha_{t+1}-\alpha_t\ge 0$.
Using $\alpha_i=\alpha_1+\sum\limits_{t=1}^{i-1}\Delta_t$ and $\sum\limits_{i=1}^r(rm_i-r_1)=0$, we obtain
\begin{align*}
w
&= \sum_{i=1}^r (rm_i-r_1)\alpha_i
= \sum_{i=1}^r \sum_{t=1}^{i-1} (rm_i-r_1)\Delta_t
= \sum_{t=1}^{r-1}\Bigl(\sum_{i=t+1}^r (rm_i-r_1)\Bigr)\Delta_t.
\end{align*}
Since
\[
\sum_{i=t+1}^r (rm_i-r_1)=r(r_1-q_t)-r_1(r-t)=r_1 t-r q_t,
\]
it follows that
\[
w=\sum_{t=1}^{r-1}(r_1 t-r q_t)\Delta_t.
\]

On the other hand,
\[
c=\sum_{j=1}^r (1-m_j)(r_1-q_j).
\]
For any fixed $t$,
\[
c \ge \sum_{j=1}^t (1-m_j)(r_1-q_j)
\ge (r_1-q_t)\sum_{j=1}^t (1-m_j)
= (r_1-q_t)(t-q_t),
\]
as $q_j\le q_t$ whenever $j\le t$. Now
\[
(r_1-q_t)(t-q_t)-(r_1 t-r q_t)=q_t(q_t+r-r_1-t)\ge 0,
\]
Indeed $t-q_t\le r-r_1$ as $t-q_t$ is the number of $m_i$'s in the set $\{ m_1, \ldots , m_t\}$ which are zero and there is exactly $r-r_1$ many $m_i$ vanishes in the set $\{ m_1, \ldots , m_r\}$. It follows that $c\ge r_1 t-r q_t$ for all $t$.

Since $\Delta_t\ge 0$, we get
\[
w=\sum_{t=1}^{r-1}(r_1 t-r q_t)\Delta_t
\le c\sum_{t=1}^{r-1}\Delta_t
= c(\alpha_r-\alpha_1) < c.
\]
Thus $w\le c$, as required.
\end{proof}

We now estimate the dimension of the space of unstable parabolic bundles with fixed determinant. Any such bundle admits a destabilizing subbundle, and hence can be described as an extension of two smaller rank bundles. Using this description, we parametrize these bundles via extension data together with compatible parabolic structures, and obtain an upper bound on their dimension. This approach goes back to the work of Sun \cite{S00} and Biswas-Bhosle \cite{BB23}, and we follow the same strategy here.

Before we proceed further, we set up some notation. We use the standard Quot-scheme construction of the parameter space of framed vector bundles, as in \cite{S00,BB23}. Let us consider the Quot scheme of quotients of $\mathcal{O}_X^{\oplus N}$ with Hilbert polynomial
\[
P(m)=rm+d+r(1-g).
\]
Let $Quot$ denote the open subset of this Quot scheme consisting of locally free quotients. Let
\[
\mathcal{O}^{\oplus N}_{Quot \times X} \longrightarrow \mathcal{E} \longrightarrow 0
\]
be the universal quotient sheaf on $Quot \times X$. Let $R \subset Quot$ be the subset of quotients $\mathcal{E}$ such that
\[
H^1(\mathcal{E})=0,\qquad H^0(\mathcal{E})=\mathbb{C}^N.
\]
Then $R$ is irreducible and non-singular, and
\[
\dim R= r^2(g-1)+1+\dim \PGL(N),
\]
see \cite{S00,BB23}. Define
\[
Q_{par}:= \prod_{x\in D,\,Quot} \mathcal{F}lag_{\overline{n}(x)}\mathcal{E}_x
\]
to be the fiber product over $Quot$ of the relative flag schemes of type $\overline{n}(x)$. Let
\[
R_{par}\longrightarrow R
\]
be the restriction of $Q_{par}$ to $R$. Since relative flag varieties are smooth and irreducible over the base, it follows that $R_{par}$ is irreducible and non-singular. Let $R_{\xi,par}\subset R_{par}$ denote the subset consisting of parabolic vector bundle quotients $\mathcal{E}$ with determinant equal to the fixed line bundle $\xi$. Similarly, let
\[
R^{ss}_{\xi,par}\subset R_{\xi,par}
\]
denote the subset corresponding to semistable parabolic bundle quotients.

\begin{prop}\label{key_dim}
Let \(X\) be a smooth projective curve of genus \(g (\geq 2)\).
Let $W$ be the set of isomorphism classes of unstable full-flag parabolic bundles $E_{*}$ of fixed rank $r$ and fixed determinant $\xi$ (of degree $d$).
Then
\[
\dim W \le (r^2-r)(g-1)-1+\frac{n(r^2-r)}{2}.
\]
\end{prop}

\begin{proof}
The argument is standard and follows closely the dimension estimates of Sun and Biswas-Bhosle (cf. \cite{S00, BB23}). We briefly recall the steps that are required in our setting.

Let $E_{1,*}$ be a parabolic subbundle of rank $r_1$ and degree $d_1$ which destabilizes $E_*$. Then we have a short exact sequence
\begin{align}\label{ses}
0 \to E_{1,*} \to E_* \to E_{2,*} \to 0,
\end{align}
with
\[
\rank(E_2)=r_2:= r-r_1,\qquad \deg(E_2)=d_2:=d-d_1,
\]
and the parabolic structure on $E_{1,*}$ is induced from that of $E_*$. 

For each point $x \in D$, let 
\[
m_1(x), m_2(x), \ldots , m_{l_x}(x)
\]
be the nonzero multiplicities for the parabolic structure on $E_{1,*}$. Note that $m_i(x) \in \{0,1\}$ for all $1\le i \le r, x \in D$ as we consider full flag parabolic structure on $E_*$. 

The destabilizing condition gives
\[
\mu_{\mathrm{par}}(E_{1,*})>\mu_{\mathrm{par}}(E_*).
\]
If we set $A:=r_2d_1-r_1d_2$, the above inequality is equivalent to
\begin{equation}\label{eq:destab}
A+\sum_{x\in D}\sum_{i=1}^r (r\,m_i(x)-r_1)\alpha_i(x)>0.
\end{equation}

\noindent \textbf{Step I: Estimate the dimension of the space of vector bundle extensions of the type \eqref{ses} without parabolic structure.} Following \cite{BB23, S00}, we consider the Quot-scheme approach.

The underlying vector bundle $E$ can be thought of as a closed point of the Quot scheme $Quot$ as described earlier. Similarly, for the bundles $E_i$, let $Q^i$ denote the Quot schemes of quotients
\[
\mathcal{O}_X^{N_i} \longrightarrow E^i \longrightarrow 0.
\]
Clearly $N_1+N_2=N$. Let
\[
\mathcal{O}^{N_i}_{Q^i \times X} \longrightarrow \mathcal{E}^i \longrightarrow 0
\]
be the universal quotient sheaf on $Q^i \times X$.

Let $Q^i_F$ be the open subset of locally free quotients with $H^1(\mathcal{E}^i)=0$ and $H^0(\mathcal{E}^i)=\mathbb{C}^{N_i}$. Also let $\mathcal{E}^i$ denote the universal quotient on $X \times Q^i_F$.

Now set $Q:= Q_F^1 \times Q_F^2$ and define
\[
\mathcal{E}:= (\mathcal{E}^2)^{\vee} \otimes \mathcal{E}^1.
\]
Let $\phi: Q \times X \longrightarrow Q$ be the natural projection map.

Then for $h \geq 0$, the locally closed subsets
\[
Q_h:= \{ y \in Q \mid h^1(\phi^{-1}(y), \mathcal{E}|_{\phi^{-1}(y)})=h\}
\]
cover $Q$.

Since $R^1\phi_*\mathcal{E}$ is locally free of rank $h$ on $Q_h$, define varieties $P_h$ as follows:
\begin{enumerate}
\item if $h=0$, we take $P_h=Q$ and $\mathcal{E}^h= \mathcal{E}^1 \oplus \mathcal{E}^2$ on $X \times P_h$.
\item if $h>0$, we define $P_h$ to be $\mathbb{P}((R^1\phi_*\mathcal{E})^{\vee})$, and $\mathcal{E}^h$ to be the universal extension
\[
0 \longrightarrow \mathcal{E}^1 \otimes \mathcal{O}_{P_h}(1) \longrightarrow \mathcal{E}^h \longrightarrow \mathcal{E}^2 \longrightarrow 0
\]
on $X \times P_h$.
\end{enumerate}

Thus the quasi-projective variety $P_h$ parametrizes vector bundle extensions of the type \eqref{ses} (without parabolic structure). The dimension of the quasi-projective subvariety $\widetilde{P}_h \subset P_h$ parametrizing vector bundle extensions of type \eqref{ses} with $\det E \cong \xi$ is given by
\[
\dim \widetilde{P}_h= \dim P_h - g,
\]
as in \cite{BB23,S00}.

\noindent \textbf{Step II: Estimate the dimension of the space of vector bundle extensions of the type \eqref{ses}.} 
For each $x \in D$, let 
\[
u(x) := \left(r_1, d_1, h, m_1(x) , \ldots ,m_{l_x}(x) \right),
\]
and set $u:= \{u(x)\mid x \in D\}$, where the data satisfy \eqref{eq:destab}. Define the locally closed subvarieties
\[
S_{u(x)} \subset \mathcal{F}lag_{\overline{n}(x)}\mathcal{E}^h_x
\]
which are fibrations over $\widetilde{P}_h$ whose fibers $S^0_{u(x)}$ consist of flags
\[
E_x=F_1(E_x) \supset \cdots \supset F_{l_x}(E_x) \supset F_{l_x+1}(E_x)=0
\]
such that
\[
\dim \bigl(F_i(E_x) \cap (E_1)_x\bigr) =r_1- \sum_{j=1}^i m_j(x).
\]

Define
\[
S_u:= \prod_{x \in D} S_{u(x)}.
\]
Then the parabolic bundle extensions of the type \eqref{ses} with $\det E \cong \xi$ are parametrized by $S_u$.

Clearly,
\begin{equation}\label{S_u dimension}
\dim S_u \leq \dim P_h - g + \sum_{x \in D} \dim S^0_{u(x)}.
\end{equation}

It is well known (cf. \cite{S00, BB23}) that
\[
\dim Q^i \leq r_i^2(g-1)+1+ \dim \PGL(N_i).
\]

Now,
\begin{equation}\label{P_h dimension}
\dim P_h \leq 
\begin{cases}
(g-1)\sum r_i^2+2+\sum \dim \PGL(N_i), & \text{if } h=0,\\[4pt]
(g-1)\sum r_i^2+2+\sum \dim \PGL(N_i)+ h-1, & \text{if } h>0.
\end{cases}
\end{equation}

The number 
\[
h= \dim \Ext^1_X(\mathcal{E}^2_y,\mathcal{E}^1_y)
\]
is computed using Riemann--Roch, and it is given by
\begin{equation}\label{h-RR}
h= 
r_1r_2(g-1)-A+h^0(X,\Hom(E_2,E_1)).
\end{equation}

At each $x\in D$, the full flag variety in $E_x$ has dimension
\[
\frac{r(r-1)}{2}.
\]
By Lemma $5.1$ of \cite{S00}, the allowed flag locus at $x$ has dimension
\begin{equation}\label{S^0_u}
\dim S^0_{u(x)} = \frac{r(r-1)}{2}-c_x.
\end{equation}

Putting \eqref{P_h dimension}, \eqref{h-RR}, and \eqref{S^0_u} into \eqref{S_u dimension}, we obtain
\begin{equation*}
\dim S_u
\le
(r^2-r_1r_2)(g-1)+1-g-A+h^0(X,\Hom(E_2,E_1))+(N_1^2+N_2^2-2)
+\frac{n(r^2-r)}{2}-\sum_{x\in D}c_x.
\end{equation*}

\noindent\textbf{Step III: Reduction to parabolic bundles.} Let us denote $F_u$ to be the frame bundle of the direct image of $\mathcal{E}^{h}$ over $S_u$; it is a principal $\GL(N)$-bundle. This yields a natural map
\[
\psi_u:F_u\to R_{\xi,\mathrm{par}}\setminus R^{ss}_{\xi,\mathrm{par}}.
\]

As $u$ varies, the union of $\psi_u(F_u)$ cover $R_{\xi,\mathrm{par}}\setminus R^{ss}_{\xi,\mathrm{par}}$. Therefore, it is enough to estimate $\dim \psi_u(F_u)$ for each $u$ (see \cite{S00}, \cite{BB23}).

Let $c$ be the infimum of the dimensions of the fibers of $\psi_u$. Since $E= \mathcal{E}^h_y$ (for $y \in P_h$) is globally generated by sections, and elements of $\Aut(E)$ act non-trivially on $H^0(E)$, the following standard estimate holds (cf. \cite{BB23}): for any vector bundle extension \eqref{ses} (without parabolic structure), one has
\[
\dim \Aut(E) \geq h^0(X,\Hom(E_2,E_1)) + 1.
\]

Hence
\[
c \geq h^0(X,\Hom(E_2,E_1)) + N_1^2 + N_2^2 - 1.
\]

Therefore,
\begin{align*}
\dim \psi_u(F_u)
&\leq \dim F_u - c \\
&= \dim S_u + N^2 - c \\
&\leq N^2+ (r^2-r_1r_2)(g-1)-g-A+\frac{n(r^2-r)}2-\sum_{x\in D}c_x.
\end{align*}
Finally,
\[
 R_{\xi,\mathrm{par}}\setminus R^{ss}_{\xi,\mathrm{par}} \, \leq \, N^2+ (r^2-r_1r_2)(g-1)-g-A+\frac{n(r^2-r)}2-\sum_{x\in D}c_x
\]
\noindent \textbf{Step IV: Estimation of $\dim W$.} 
The isomorphism classes of parabolic bundles in $W$ are obtained by forgetting the choice of framing, which amounts to quotienting by the action of $\PGL(N)$. Hence,
\begin{equation}\label{dimension_1}
\dim W \leq (r^2-r_1r_2)(g-1)+1-g-A+\frac{n(r^2-r)}2-\sum_{x\in D}c_x.
\end{equation}
We now estimate the final term in the right-hand side of the preceding inequality using \eqref{eq:destab}.For each $x \in D$, set
\[
w_x:=\sum_{i=1}^r (r\,m_i(x)-r_1)\alpha_i(x).
\]

By Lemma~\ref{key_lemma}, applied pointwise on $D$, we have
\begin{equation}\label{eq:wx-cx}
w_x\leq c_x
\qquad\text{for every }x\in D.
\end{equation}
Combining \eqref{eq:wx-cx} with the destabilizing inequality \eqref{eq:destab}, we obtain
\[
A+\sum_{x\in D}c_x \ge A+\sum_{x\in D}w_x>0.
\]
Since the left-hand side is an integer, it follows that
\[
A+\sum_{x\in D}c_x\ge 1.
\]
We substitute this bound into \eqref{dimension_1} to obtain
\[
\dim W
\le
(r^2-r_1r_2)(g-1)-g+\frac{n(r^2-r)}{2}.
\]
Finally, using the inequality \(r_1r_2\ge r-1\), it is concluded that
\[
\dim W
\le
(r^2-r)(g-1)-1+\frac{n(r^2-r)}{2}.
\]
\end{proof}

 	We are now ready to estimate the codimension of the complement of the open set $\mathcal{U}_{\mathrm{conn}}(\alpha, r,\xi)$ in $\mathcal{M}_{\mathrm{conn}}(\alpha, r,\xi)$. Recall that $\mathcal{U}_{\mathrm{conn}}(\alpha, r,\xi)$  parametrizes all stable parabolic connection $(E_*, \nabla)$ such that the underlying parabolic bundle $E_*$ is stable. The codimension estimation was provided in \cite{BB23}, \cite[Section $3$]{MY20}, \cite[Section $5$]{S00} (also see \cite{B99, BM07} for non-parabolic case); however, we present here a refined version to derive our result.

       \begin{prop}\label{codim}
       Let $\mathcal{S} \subset \mathcal{M}_{\mathrm{conn}}(\alpha, r ,\xi)$ be the subspace of stable parabolic connections $(E_*,\nabla)$ such that $E_*$ is not stable, i.e.
  \[
  \mathcal{S} = \mathcal{M}_{\mathrm{conn}}(\alpha, r ,\xi) \setminus \mathcal{U}_{\mathrm{conn}}(\alpha, r ,\xi).
  \]
  Then the codimension of $\mathcal{S}$ in $\mathcal{M}_{\mathrm{conn}}(\alpha, r ,\xi)$ is given by
           \[
           \mathrm{codim}(\mathcal{S}) \geq (r-1)(g-1)+1.
           \]
           
           In particular, if $r\geq 2, g\geq 2$ then $\mathrm{codim}(\mathcal{S}) \geq 2$.
       \end{prop}
       \begin{proof}
        Consider any element $(E_*,\nabla) \in \mathcal{S}$ such that the underlying parabolic bundle $E_*$ is not stable. Since the weights are generic, it implies that $E_*$ is not semistable. Let 
           \[
           0 = E_*^0 \subset E_*^1 \subset E_*^2 \subset \cdots \subset E_*^{l-1} \subset E_*^l = E_*
           \]
           be the Harder-Narasimhan filtration of $E_*$. The set of all pairs $\{(\mathrm{rank}(E_*^i),\mathrm{pardeg}(E_*^i))\}_{i=1}^l \in \mathbb{Z} \times \mathbb{R}$ is called the \textit{Harder-Narasimhan polygon} of $E_*$ (see \cite[P. $173$]{S77}). Let $\mathcal{W}$ be the space of all isomorphism classes of parabolic bundles with fixed determinant $\xi$ whose  Harder-Narasimhan polygon (of length $l \ge 2$) coincides with that of the given parabolic bundle $E_*$. Then all the elements in $\mathcal{W}$ have Harder-Narasimhan filtration of the same length $l$; they are all unstable, in particular. In other words, $\mathcal{W}$ is contained in the set $W$, described as in \Cref{key_dim}. The dimension estimation of \Cref{key_dim} yields  
          
\[
\dim \mathcal{W} \leq (r^2-r)(g-1)-1+\frac{n(r^2-r)}{2}.
\]

           Let $\mathcal{V}_\mathrm{HN}(E_*)$ denote the subvariety of $\mathcal{M}_{\mathrm{conn}}(\alpha, r, \xi)$ which consists of all pairs of the form $(E_*',\nabla') \in \mathcal{M}_{\mathrm{conn}}(\alpha, r, \xi)$ whose Harder-Narasimhan polygon coincides with that of $E_*$. By a similar argument as in \cite[Proposition $1.9$]{N86}, it follows that this subset of $\mathcal{M}_{\mathrm{conn}}(\alpha, r, \xi)$ is constructible (cf. \cite[Proposition $10$, p. $182$]{S77}). Also
           \begin{align}\label{maximum}
          \begin{split}
             \dim \mathcal{V}_\mathrm{HN}(E_*) & \leq \dim \mathcal{W}+ \mathrm{dim} \big(\mathcal{P}\textit{Conn}(E_*)/\mathrm{Aut}(E_*)\big) \\
             &\leq (r^2-r)(g-1)-1+\frac{n(r^2-r)}{2} + (r^2-1)(g-1) + \frac{n(r^2-r)}{2}\\
             &= (2r^2-r-1)(g-1) -1 +n(r^2-r).
          \end{split}
          \end{align}
          Observe that the final bound does not depend on $E_*$. The similar argument as in \cite[Proposition $11$, p. $183$]{S77} yields that there are only finitely many Harder-Narasimhan polygons that appear for the parabolic bundles on $X$
in the bounded family over $\mathcal{M}_{\mathrm{conn}}(\alpha, r, \xi)$. Let $\mathcal{V}_\mathrm{HN}$ denote the union of subvarieties of the form $\mathcal{V}_\mathrm{HN}(E_*)$, where $(E_*,\nabla) \in \mathcal{S}$. Then 
\begin{align*}
    \dim \mathcal{V}_\mathrm{HN} \leq (2r^2-r-1)(g-1) -1 +n(r^2-r).
\end{align*}
Since $\dim \mathcal{M}_{\mathrm{conn}}(\alpha, r, \xi) = 2(r^2-1)(g-1) + n(r^2-r)$, we get 
          \begin{equation*}
              \dim \mathcal{M}_{\mathrm{conn}}(\alpha, r, \xi) - [(2r^2-r-1)(g-1) -1 +n(r^2-r)] = (r-1)(g-1) + 1.
          \end{equation*}
          Since $\mathcal{S} \subseteq \mathcal{V}_\mathrm{HN}$, we get
          \[
          \mathrm{codim}(\mathcal{S}) \geq \mathrm{codim}(\mathcal{V}_\mathrm{HN}) \geq (r-1)(g-1)+1.
          \]  
       \end{proof}

	\section{Fundamental Group and Homotopy Groups}
	\begin{theorem}\label{main}
		Let $X$ be a compact Riemann surface of genus $g \geq 2$ and let $\xi$ be a line bundle over $X$. Then the homotopy groups of the moduli space $\mathcal{M}(\alpha, r,\xi)$ of stable parabolic bundles of rank $r$ $(\geq 2)$ and determinant $\xi$ and the moduli space $\mathcal{M}_{\mathrm{conn}}(\alpha, r,\xi)$ of stable parabolic connections of rank $r$ $(\geq 2)$ and determinant $\xi$ are isomorphic, i.e.
  \[
  \pi_k(\mathcal{M}_{\mathrm{conn}}(\alpha, r, \xi)) \cong \pi_k(\mathcal{M}(\alpha, r, \xi))
  \]
  for all $k=1,\dots, 2(r-1)(g-1)$.
	\end{theorem}
	\begin{proof}
	 We will show that
		\[
		\mathcal{U}_{\mathrm{conn}}(\alpha, r, \xi) = \mathcal{M}_{\mathrm{conn}}(\alpha, r, \xi) \setminus \mathcal{S}.
		\]
		 is a torsor over $\mathcal{M}(\alpha, r, \xi)$ and the fibers are contractible, so the moduli space $\mathcal{M}(\alpha, r, \xi)$ is a deformation retracts of $\mathcal{U}_{\mathrm{conn}}(\alpha, r, \xi)$. Consider the forgetful map $\pi$ given as in \eqref{forgetful}. By \Cref{dimension1}, it follows that the dimension of the fiber $\pi^{-1}(E_*)$ is constant, and thus the map $\pi$ is an equidimensional surjective morphism. From the deformation theory, we know that the tangent space of $\mathcal{M}(\alpha, r, \xi)$ at the point $E_*$ is given by
  \[
  T_{E_*}\mathcal{M}(\alpha, r, \xi) \cong H^1(X, \mathrm{PEnd}_0(E_*)).
  \]
By the parabolic version of the Serre duality, we get
  \[
  H^1(X, \mathrm{PEnd}_0(E_*)) \cong H^0(X, \mathrm{SPEnd}_0(E_*) \otimes K(D))^\vee.
  \]
  Thus the cotangent space of $\mathcal{M}(\alpha, r, \xi)$ at $E_*$ is given by 
  \[
  T^*_{E_*}\mathcal{M}(\alpha, r, \xi) \cong H^0(X, \mathrm{SPEnd}_0(E_*) \otimes K(D))
  \]
  Since the cotangent space $T^*_{E_*}\mathcal{M}(\alpha, r, \xi)$ acts on the fiber $\pi^{-1}(E_*)$ freely transitively, we conclude that $$ \pi : \mathcal{U}_{\mathrm{conn}}(\alpha, r, \xi) \longrightarrow \mathcal{M}(\alpha, r, \xi)$$ is a torsor over $\mathcal{M}(\alpha, r, \xi)$. Thus the fibers $\pi^{-1}(E_*)$ are connected and contractible. Consider the following exact sequence of homotopy groups
  \begin{equation}\label{exactseq}
  \cdots \rightarrow \pi_k(\pi^{-1}(E_*)) \rightarrow  \pi_k(\mathcal{U}_{\mathrm{conn}}(\alpha, r, \xi)) \rightarrow \pi_k(\mathcal{M}(\alpha, r, \xi)) \rightarrow   \pi_{k-1}(\pi^{-1}(E_*)) \rightarrow \cdots. 
  \end{equation}
 For $k=1$, since $\pi_1(\mathcal{M}(\alpha, r, \xi))$ is trivial (in fact, $\mathcal{M}(\alpha, r, \xi)$ rational as follows from \cite{BY99}) and the fibers are contractible, from the above sequence \eqref{exactseq} it follows that $\pi_1(\mathcal{U}_{\mathrm{conn}}(\alpha, r, \xi)) $ is also trivial. Thus, we get
  \[
    \pi_k(\mathcal{U}_{\mathrm{conn}}(\alpha, r, \xi)) \cong \pi_k(\mathcal{M}(\alpha, r, \xi))
  \]
  for all $k\geq 1$.

  Since the moduli space $\mathcal{M}_{\mathrm{conn}}(\alpha, r, \xi)$ is a smooth variety and by  \Cref{codim}, the codimension of the complement of $\mathcal{U}_{\mathrm{conn}}(\alpha, r, \xi)$ is given by $$\mathrm{codim}(\mathcal{S}) \geq (r-1)(g-1)+1,$$ we conclude that
  \[
  \pi_k(\mathcal{M}_{\mathrm{conn}}(\alpha, r, \xi)) \cong \pi_k(\mathcal{U}_{\mathrm{conn}}(\alpha, r, \xi)) \cong \pi_k(\mathcal{M}(\alpha, r, \xi))
  \]
  for every $k=1,\dots ,2(r-1)(g-1)$. 
\end{proof}
 \begin{corollary}\label{cor}
     Let $X$ be a compact Riemann surface of genus $g \geq 2$ and let $\xi$ be a line bundle over $X$. Then the moduli space $\mathcal{M}_{\mathrm{conn}}(\alpha, r,\xi)$ is simply connected.
     
     Moreover, if either $r \geq 3$ or $g \geq 3$, then we have
     \[
     \pi_2(\mathcal{M}_{\mathrm{conn}}(\alpha, r, \xi)) \cong H_2(\mathcal{M}(\alpha, r,\xi)).
     \]
 \end{corollary}
 \begin{proof}
    Since the moduli space $\mathcal{M}(\alpha, r,\xi)$ is simply connected, applying \Cref{main}, we conclude that $\mathcal{M}_{\mathrm{conn}}(\alpha, r, \xi))$ is simply connected as well. 

    For the second part, the assumptions on $r$ and $g$ produce
    \[
    2(r-1)(g-1) \geq 4.
    \]
Since $\pi_1(\mathcal{M}(\alpha, r,\xi))$ is trivial, the Hurewicz map
      \[
      \pi_2(\mathcal{M}(\alpha, r,\xi)) \longrightarrow H_2(\mathcal{M}(\alpha, r,\xi))
      \]
      is an isomorphism. Therefore, by \Cref{main}, we get
      \[
      \pi_2(\mathcal{M}_{\mathrm{conn}}(\alpha, r, \xi))) \cong H_2(\mathcal{M}(\alpha, r,\xi)).
      \]
\end{proof}

\begin{remark}
    Following a similar argument, we can also conclude that the moduli space of stable parabolic Higgs bundles over $X$ with determinant $\xi$ is simply connected. 
\end{remark}

\section{Pure and Mixed Hodge Structures}\label{mhs}
According to Deligne, \cite{D72, D75}, for any complex algebraic variety $Y$, the torsion-free part of the $i$-th cohomology space $H^i(Y,\mathbb{Z})$ with $\mathbb{Z}$-coefficients carries two finite filtrations by subspaces for all $i \geq 0$, the Hodge filtration $F$ and the weight filtration $W$. We will first recall the definition of mixed Hodge structures of a complex variety.

\begin{definition}
    A \textit{pure Hodge structure} of weight $k \in \mathbb{Z}$ consists of an abelian group $H_\mathbb{Z}$ and a decomposition of $H \coloneqq H_\mathbb{Z} \otimes \mathbb{C}$ into a direct sum of complex subspaces
    \[
    H = \bigoplus_{p+q=k} H^{p,q},
    \]
    such that $H^{q,p} = \overline{H^{p,q}}$.
\end{definition}
  \begin{definition}
      A \textit{mixed Hodge structure} on an abelian group $H_\mathbb{Z}$ is equipped with:
      \begin{enumerate}
          \item \textnormal{the Hodge filtration: a finite descending filtration}
          \[
          H = H_\mathbb{Z}\otimes \mathbb{C} = F^0 \supset F^1 \supset F^2 \cdots \supset F^m=0
          \]
          \item \textnormal{the weight filtration: a finite ascending filtration of} $H_\mathbb{Q} \coloneqq H_\mathbb{Z}\otimes \mathbb{Q}$
          \[
          0 =W^0 \subset W^1 \subset \cdots \subset W^l = H_\mathbb{Q},
          \]
          such that $F$ induces a pure Hodge structure of weight $k$ on every $$\mathrm{Gr}^k_W(H_\mathbb{Q}) = W^k/W^{k-1}.$$
      \end{enumerate}
  \end{definition}  
For any quasi-projective algebraic variety $Y$, the torsion-free part of the cohomology spaces $H^k(Y,\mathbb{Z})$ has mixed Hodge structure and if $Y$ is smooth and projective then the Hodge structure is pure of weight $k$ (cf. \cite{D72, D75}). We will show that the mixed Hodge structure on the torsion-free part of the $i$-th cohomology group $H^i(\mathcal{M}_{\mathrm{conn}}(\alpha, r, \xi)), \mathbb{Z})$ of the moduli space of stable parabolic connections $\mathcal{M}_{\mathrm{conn}}(\alpha, r, \xi))$ is pure of weight $i$ for all $i\leq 2(r-1)(g-1)+1$.

Let $\mathrm{T}^i(\mathcal{M}_{\mathrm{conn}})$ (respectively, $\mathrm{T}^i(\mathcal{M})$) denote the torsion of the $i$-th cohomology group $H^i(\mathcal{M}_{\mathrm{conn}}(\alpha, r, \xi),\mathbb{Z})$ (respectively, $H^i(\mathcal{M}(\alpha, r, \xi),\mathbb{Z})$,  where $\mathcal{M}(\alpha, r, \xi)$ is the moduli space of stable parabolic bundles of rank $r (\geq 2)$ over $X$ with fixed determinant $\xi$.
\begin{theorem}\label{mhstheorem}
Let $X$ be a compact Riemann surface of genus $g \geq 2$. The mixed Hodge structure on $H^i(\mathcal{M}_{\mathrm{conn}}(\alpha, r, \xi),\mathbb{Z})/\mathrm{T}^i(\mathcal{M}_{\mathrm{conn}})$ is isomorphic to the mixed Hodge structure on $H^i(\mathcal{M}(\alpha, r, \xi),\mathbb{Z})/\mathrm{T}^i(\mathcal{M})$ for all $i \leq 2(r-1)(g-1) + 1$.

\end{theorem}
\begin{proof}
    Since $$\pi : \mathcal{U}_{\mathrm{conn}}(\alpha, r, \xi) \longrightarrow \mathcal{M}(\alpha, r, \xi)$$ is a torsor with contractible fibers, the induced morphism
    \begin{equation}\label{isom}
        \pi^* : H^i(\mathcal{M}(\alpha, r, \xi), \mathbb{Z}) \longrightarrow H^i(\mathcal{U}_{\mathrm{conn}}(\alpha, r, \xi), \mathbb{Z})
    \end{equation}
    is an isomorphism for all $i \geq 0$. This gives an isomorphism of two mixed Hodge structures of $H^i(\mathcal{M}(\alpha, r, \xi), \mathbb{Z})/\mathrm{T}^i(\mathcal{M})$ and $H^i(\mathcal{U}_{\mathrm{conn}}(\alpha, r, \xi), \mathbb{Z})/\mathrm{T}^i(\mathcal{U}_{\mathrm{conn}})$, where $\mathrm{T}^i(\mathcal{U}_{\mathrm{conn}})$ is the torsion part of $H^i(\mathcal{U}_{\mathrm{conn}}(\alpha, r, \xi), \mathbb{Z})$. The Hodge structure on $H^k(\mathcal{M}(\alpha, r, \xi), \mathbb{Z})/\mathrm{T}^k(\mathcal{M})$ is pure of weight $k$ as $\mathcal{M}(\alpha, r, \xi)$ is a smooth projective variety. Therefore, the Hodge structure on $H^k(\mathcal{U}_{\mathrm{conn}}(\alpha, r, \xi), \mathbb{Z})/\mathrm{T}^k(\mathcal{U}_{\mathrm{conn}})$ is also pure of weight $k$. We now consider the pair $$(\mathcal{M}_{\mathrm{conn}}(\alpha, r, \xi), \mathcal{U}_{\mathrm{conn}}(\alpha, r, \xi)).$$ Then the relative cohomologies of this pair give a long exact sequence
    \begin{alignat}{2}\label{les}
    \begin{split}
    \cdots &\rightarrow H^i(\mathcal{M}_{\mathrm{conn}}(\alpha, r, \xi), \mathcal{U}_{\mathrm{conn}}(\alpha, r, \xi), \mathbb{Z}) \rightarrow H^i(\mathcal{M}_{\mathrm{conn}}(\alpha, r, \xi), \mathbb{Z}) \xrightarrow{\varphi} H^i(\mathcal{U}_{\mathrm{conn}}(\alpha, r, \xi), \mathbb{Z}) \rightarrow \\ &\rightarrow H^{i+1}(\mathcal{M}_{\mathrm{conn}}(\alpha, r, \xi), \mathcal{U}_{\mathrm{conn}}(\alpha, r, \xi), \mathbb{Z}) \rightarrow \cdots, 
    \end{split}
    \end{alignat}
where $\varphi$ is induced from the inclusion map
\[
\mathrm{incl} :  \mathcal{U}_{\mathrm{conn}}(\alpha, r, \xi) \longrightarrow \mathcal{M}_{\mathrm{conn}}(\alpha, r, \xi).
\]
By \cite[p. $43$]{D75}, $\varphi$ is a homomorphism of mixed Hodge structures. From \Cref{codim}, it follows that the (complex) codimension 
\[
\mathrm{codim}(\mathcal{M}_{\mathrm{conn}}(\alpha, r, \xi) \setminus \mathcal{U}_{\mathrm{conn}}(\alpha, r, \xi)) \geq (r-1)(g-1) +1,
\]
Therefore,
\[
H^i(\mathcal{M}_{\mathrm{conn}}(\alpha, r, \xi), \mathcal{U}_{\mathrm{conn}}(\alpha, r, \xi), \mathbb{Z}) = 0
\]
for all $i \leq 2(r-1)(g-1) +2$. Thus, the homomorphism $\varphi$ in the sequence \eqref{les} is an isomorphism. Therefore, the composition
\[
\varphi^{-1} \circ \pi^* : H^i(\mathcal{M}(\alpha, r, \xi), \mathbb{Z}) \longrightarrow H^i(\mathcal{M}_{\mathrm{conn}}(\alpha, r, \xi), \mathbb{Z})
\]
is an isomorphism for all $i\leq 2(r-1)(g-1) +1$, where $\pi^*$ is as in \eqref{isom}. This gives the required isomorphism
\[
H^i(\mathcal{M}_{\mathrm{conn}}(\alpha, r, \xi),\mathbb{Z})/\mathrm{T}^i(\mathcal{M}_{\mathrm{conn}}) \cong H^i(\mathcal{M}(\alpha, r, \xi),\mathbb{Z})/\mathrm{T}^i(\mathcal{M})
\]
of mixed Hodge structures for all $i\leq 2(r-1)(g-1) +1$. 
\end{proof}

\section*{Acknowledgment}
The authors are grateful to Indranil Biswas for several helpful discussions. Both authors are supported by the INSPIRE Faculty Fellowship (Ref. Nos. IFA21-MA 161 and IFA22-MA 186) funded by the Department of Science and Technology, Government of India. The first author thanks the Indian Statistical Institute, Kolkata, for providing a stimulating research environment where a preliminary version of this work was completed.

\bibliography{ref} 
\bibliographystyle{siam}	
	
\end{document}